\documentclass[10pt]{amsart}
\usepackage[utf8]{inputenc}
\usepackage{amsfonts}
\usepackage{graphics}
\usepackage{amsxtra,amsmath,amscd,amssymb,amsthm}
\usepackage{mathrsfs}
\usepackage{marginnote}
\usepackage{longtable}
\usepackage{verbatim}
\usepackage{xspace}
\usepackage{ifthen}
\usepackage{tabularx}
\usepackage{setspace,graphicx}
\usepackage{multirow,xcolor}
\usepackage{color, colortbl,enumerate}
\usepackage{tikz}
\usepackage{tikz-cd}
\usepackage{leftidx}
\usepackage[titletoc]{appendix}
\usepackage{cite}
\usepackage{floatrow}
\newcommand{\quash}[1]{}  

\theoremstyle{plain}
\numberwithin{equation}{section}

\newtheorem{thm}[equation]{Theorem}
\newtheorem{prop}[equation]{Proposition}
\newtheorem{lm}[equation]{Lemma}
\newtheorem{defi}[equation]{Definition}

\newtheorem{rk}[equation]{Remark}

\newtheorem{ex}[equation]{Example}

\theoremstyle{remark}

\newcommand{\Gortz}{G\"{o}rtz}

\renewcommand{\phi}{\varphi}

\newcommand{\Isom}{{\rm Isom}}

\newcommand{\Hom}{{\rm Hom}}

\newcommand{\Spec}{{\rm Spec}}

\newcommand{\etale}{\' etale }

\newcommand{\Gr}{{\rm Gr}}

\newcommand{\Spin}{{\rm Spin}}
\newcommand{\GL}{{\rm GL}}
\newcommand{\SL}{{\rm SL}}

\newcommand{\SO}{{\rm SO}}
\newcommand{\ord}{{\rm ord}}

\newcommand{\GGL}{{\rm\bf GL}}

\newcommand{\SSO}{{\rm\bf SO}}

\newcommand{\SSpin}{{\rm\bf Spin}}

\newcommand{\RRT}{{\rm\bf RT}}
\newcommand{\AAut}{{\rm\bf Aut}}

\newcommand{\ZZ}{\mathbb{Z}}

\newcommand{\LL}{\mathbb{L}}

\newcommand{\RR}{\mathbb{R}}

\newcommand{\calO}{\mathcal{O}}

\newcommand{\calF}{\mathcal{F}}
\newcommand{\calL}{\mathcal{L}}

\newcommand{\calB}{\mathcal{B}}
\newcommand{\calM}{\mathcal{M}}
\newcommand{\calR}{\mathcal{R}}

\newcommand{\scrG}{\mathscr{G}}

\newcommand{\scrF}{\mathscr{F}}

\newcommand{\bb }{\langle}
\newcommand{\pp}{\rangle}
\newcommand{\llp}{(\!(}
\newcommand{\rlp}{)\!)}
\newcommand{\lp}{[\![}
\newcommand{\rp}{]\!]}

\newcommand{\mathleft}{\@fleqntrue\@mathmargin0pt}
\newcommand{\mathcenter}{\@fleqnfalse}

\newcommand{\Addresses}{{
		\bigskip
		\footnotesize
		
		\textsc{Morningside Center of Mathematics, Chinese Academy of Sciences,
			Beijing, 100190, China}\par\nopagebreak
		\textit{E-mail address:} \texttt{zhihaozhao@amss.ac.cn}
}}		
\begin{document}
	\bibliographystyle{plain}

	\title{Affine Grassmannians for $G_2$}
	\date{}
	\author{Zhihao Zhao}
	\maketitle

	\begin{abstract}
	We study affine Grassmannians for the exceptional group of type $G_2$. This group can be given as automorphisms of octonion algebras (or para-octonion algebras). By using this automorphism group, we consider all maximal parahoric subgroups in $G_2$, and give a description of affine Grassmannians for $G_2$ as functors classifying suitable orders in a fixed space.
	  
	\end{abstract}
	
	\tableofcontents
	
	\section{Introduction}\label{Intro}
	
	In this paper we give an explicit description of affine Grassmannians for the exceptional group  $G_2$. Generally, affine Grassmannians can be defined by loop groups. Let $k$ be a perfect field, with char$(k)\neq 2$, and let $G$ be a linear algebraic group over $\Spec(k\llp t\rlp)$, where $k\llp t\rlp$ is the Laurent power series with variable $t$. In \cite{BT1}, \cite{BT2}, Bruhat-Tits attach to any connected reductive group $G$ over $\Spec(k\llp t\rlp)$ its building $\calB(G)$, which is a polysimplicial complex equipped with an action of $G(k\llp t\rlp)$. For each point $x\in\calB(G)$, there is a unique smooth affine  group scheme $\scrG_x$ over $\Spec(k\lp t\rp)$, where $k\lp t\rp$ is the formal power series with variable $t$. We call it the parahoric subgroup associated to the point $x$. The $k\lp t\rp$-points $\scrG_x(k\lp t\rp)$ is (at least when $G$ is simply connected) the stabilizer  of $x$ in $G(k\llp t\rlp)$. Its generic fiber is equal to $G$. We define the algebraic loop group $LG$ associated to $G$, which is the ind-scheme representing the functor:
	\[
	R\mapsto LG(R)=G(R\llp t\rlp),
	\]
	for any $k$-algebra $R$. Since $R\llp t\rlp$ is a $k\llp t\rlp$-algebra, this definition makes sense. We also define the positive loop group $L^+\scrG_x$ associated to $\scrG_x$, representing the functor:
	\[
	R\mapsto L^+\scrG_x(R)=\scrG_x(R\lp t\rp).
	\]
	for any $k$-algebra $R$. Thus, $L^+\scrG_x\subset LG$ is a subgroup functor, and the fpqc-quotient 
	\[
	\Gr_{\scrG_x}:=LG/L^+\scrG_x
	\]
	is by definition the affine Grassmannian associated to $\scrG_x$. This definition is given by Pappas-Rapoport in \cite{PR3}, where they develop a theory of twisted loop groups and of their associated flag varieties. We refer \cite{BL}, \cite{BD} and \cite{PR3} for important results in the general theory of affine Grassmannians.
		
	We now take $G$ to be the exceptional group $G_2$ and study its corresponding affine Grassmannians. Our goal (see Theorem \ref{thm 11}, \ref{thm 12}) is to give an explicit description of these affine Grassmannians in terms of lattices with certain extra conditions. A sketch of the history of this problem is the following: Lusztig in \cite{Lu} first shows that affine Grassmannians for simple Lie algebras can be described in terms of lattices that are closed under the Lie bracket. Our goal here is more in line of which is known for classical groups. The prototype is the general linear group $\GGL_n$. It is well known that affine Grassmannians for $\GGL_n$ can be viewed as the space of lattices in a fixed $n$-dimensional vector space. For other classical groups, Pappas-Rapoport give a description of affine Grassmannians and affine flag varieties for unitary groups in \cite{PR3} using lattices (or lattice chains) which are self-dual for a hermitian form; \Gortz~ considers the symplectic group case in \cite{Go2}; Smithling works out the case of split orthogonal groups in \cite{Sm2}.
	
	 In \cite{Zhao2}, the author describes affine Grassmannians for triality groups, which are groups of type $^3D_4$. It turns out that our results here are similar to \cite{Zhao2}, and different from other classical group cases. One reason is that the exceptional group $G_2$ is constructed by octonion algebras, instead of usual vector spaces with quadratic forms. Since triality groups can be constructed by certain twisted composition algebras, its structure is similar to the exceptional group $G_2$. 
	 An explicit description of the building for triality groups does not seem to be known, so in \cite{Zhao2}, we only described affine Grassmannians for the ``standard parahoric subgroup" (the parahoric subgroup which is the stabilizer of the standard lattice). In this paper, we can describe affine Grassmannians for all maximal parahoric subgroups in $G_2$, since the building of $G_2$ is already described by Gan-Yu in \cite{GY1}.
	
	 To explain our results, we need to introduce some notation. Let $C$ be an $8$-dim algebra equipped with a nonsingular quadratic form $q$, satisfying $q(x\cdot y)=q(x)q(y)$. We call $C$ an octonion algebra. Generally, algebras satisfying $q(x\cdot y)=q(x)q(y)$ are called composition algebras. An octonion algebra $C$ is a composition algebra with highest dimension by the classification theorem (see Theorem \ref{thm 25} below). Springer in \cite[\S 2.3]{Sp} shows that the automorphism group $\AAut(C)$ is of type $G_2$. We will see that there is a natural embedding of $\AAut(C)$ to the orthogonal group $\SSO(C)$, since the multiplication of $C$ determines the quadratic form $q$. In \cite{GY1}, Gan-Yu show that this embedding gives rise to a canonical embedding of buildings $\calB(\AAut(C))\rightarrow \calB(\SSO(C))$.
	 
	\quash{
	The exceptional group $G_2$ can be viewed as an automorphism group of (split) octonion algebras. These algebras are neither commutative nor associative.
	}	
	Another ingredient we need is a parahoric subgroup $\scrG_x$ associated to $x\in \calB(G_2)$. In \cite{GY1}, Gan-Yu study the building of $G_2$, and give a nice classification of all its vertices (see \cite[Theorem 9.5]{GY1}). There are three types of maximal parahoric subgroups corresponding to three types of vertices in $\calB(G_2)$ (see Theorem \ref{thm 35} below). The first type consists of the parahoric hyperspecial subgroups associated to hyperspecial points. By fixing this parahoric subgroup $\scrG=\scrG_x$, with the generic fiber $\scrG_\eta=G=G_2$, our first main theorem is the following:
	 
\begin{thm}\label{thm 11}
	    There is an $LG$-equivariant isomorphism
		\[
		\Gr_{\scrG}\simeq \mathscr{F},
		\]	
		where the functor $\mathscr{F}$ sends a $k$-algebra $R$ to the set of finitely generated projective $R\lp t\rp $-modules $L$ (i.e., $R\lp t\rp $-lattices) of $C_s\otimes_{k\llp t\rlp} R\llp t\rlp$ such that
		\begin{itemize}
			\item [(1)] $L$ is self dual under the bilinear form $\bb\ ,\ \pp$, i.e., the form induces an isomorphism $L\simeq \Hom_{R\lp t\rp}(L, R\lp t\rp )$.
			\item [(2)] $L$ is an order in $C_s\otimes_{k\llp t\rlp} R\llp t\rlp$, i.e.,  $e\in L$, $L\star L\subset L$.
	    \end{itemize}
\end{thm}

Here $C_s$ is the split para-octonion algebra over $k\llp t\rlp$ obtained from $C$, and $\star$ is the multiplication in $C_s$. The element $e\in C_s$ is the para-unit. We refer \S \ref{Octonion} for more details. This theorem is shown in \S \ref{Definition}. 	

For the other two types of maximal parahoric subgroups, we have similar results. By fixing parahoric subgroups $\scrG_2$ (resp. $\scrG_3$) in $G_2$ (see Definition \ref{def 63}), our second main theorem is the following:
\quash{
If $F=k\llp t\rlp$ is a local field, after passing to a sufficiently big unramified extension $F'=k'\llp t\rlp$ of $F$, we consider the split para-octonion algebra over $F'$, still denoted by $C_s$.
} 
\begin{thm}\label{thm 12}
	$(a)$. There is an $LG$-equivariant isomorphism 
	\[
	\Gr_{\scrG_2}\simeq \calF_2
	\] 
	where the functor $\calF_2$ sends a $k$-algebra $R$ to the set of $R\lp t\rp$-lattices $L$ of $C_s \otimes_{k\llp t\rlp} R\llp t\rlp$ such that 
	\begin{itemize}
			\item [(1)] $L$ is an order in $C_s \otimes_{k\llp t\rlp} R\llp t\rlp$, i.e., $e\in L$, $L\star L\subset L$.
			\item [(2)] $L\subsetneq L^{\vee}\subsetneq t^{-1}L$.
			\item [(3)] $L^{\vee}\star L^{\vee}\subset t^{-1}L$. 
	    \end{itemize}
$(b)$. There is an $LG$-equivariant isomorphism 
	\[
	\Gr_{\scrG_3}\simeq \calF_3
	\] 
	where the functor $\calF_3$ sends a $k$-algebra $R$ to the set of $R\lp t\rp$-lattices $L$ of $C_s \otimes_{k\llp t\rlp} R\llp t\rlp$ such that 
	\begin{itemize}
			\item [(1)] $L$ is an order in $C_s \otimes_{k\llp t\rlp} R\llp t\rlp$.
			\item [(2)] $L\subsetneq L^{\vee}\subsetneq t^{-1}L$.
			\item [(3)] $M:=tL^{\vee}\star L^{\vee}+L$ is self-dual. 
	    \end{itemize}
 \end{thm}

Here $L^\vee=\{x\in C_s \otimes_{k\llp t\rlp} R\llp t\rlp\mid \bb x,L\pp\subset R\lp t\rp\}$ is the dual lattice of $L$.  Our second main theorem is shown in \S \ref{AffGrass2}. 

The proof of Theorem \ref{thm 11} is similar to the proof in \cite{Zhao2}. That is, by decomposing a lattice $L\in \calF$ satisfying certain conditions as above, we can find a basis of $L$ such that the multiplication table of this basis is the same as the multiplication table of the ``standard lattice" $\LL$. This heavily rely on the fact that $L$ is self dual. In Theorem \ref{thm 12}, when $L\in \calF_2$ or $\calF_3$ is not self dual, we need to find a ``standard basis" by using the classification of quadratic forms in \cite[Appendix]{RZ}.

Octonions algebras play an important role in the study of other exceptional groups. For instance, we can construct Albert algebras by regarding their elements as $3\times 3$ matrices $(x_{ij})$ with $x_{ij}$ in the field $k$ or in $C$. The automorphism group of an Albert algebra is of type $F_4$, and the group of transformations that leave the cubic form $\det$ invariant is of type $E_6$ (see \cite[Ch. 7]{Sp}). We hope to use the work  here to obtain explicit descriptions of affine Grassmannians for other exceptional groups.
	   
The organization of the paper is as follows: In \S \ref{Octonion},  we review definitions and basic propositions of composition algebras. We are particular interested in composition algebras of dimension 8, which are octonion algebras and para-octonion algebras. We give a construction of octonion algebras, and consider their isotropic subspaces. In \S \ref{Building}, we recall the building theory of group $G_2$. We first show that there is an action of $S_3$ on the spin group of the (split) para-octonion algebra, and the subgroup of this spin group, which is fixed under $S_3$, is the automorphism group $G$ of type $G_2$ (see Proposition \ref{prop 32}). By using this relation, we study the building of $G$ and classify vertices in $\calB(G)$. These vertices corresponds to the maximal parahoric subgroups. We refer \cite{GY1} for more details. In \S \ref{AffGrass1}, we construct affine Grassmanninas for $\scrG$, and give our first main theorem. The proof of Theorem \ref{thm 11} is in \S \ref{ProofofTheorem}. In the last section, we fix the other two maximal parahoric subgroups $\scrG_2$ (resp. $\scrG_3$) and construct affine Grassmannians $\Gr_{\scrG_2}$ (resp. $\Gr_{\scrG_3}$). By using the classification of quadratic forms over local fields (see \cite{Ge}), we prove Theorem \ref{thm 12}.

{\bf Acknowledgement}. I thank G. Pappas for several useful
comments and corrections that greatly improved this paper.


\section{Octonion Algebras and Para-Octonion Algebras}\label{Octonion}	

In this section, we will introduce unital composition algebras and symmetric composition algebras. These algebras are often not commutative or associative. We refer \cite[Ch. VIII]{KMRT}, \cite[Ch. 1]{Sp} for general facts and more details. The octonion algebra (resp. para-Octonion algebra) is a unital composition algebra (resp. symmetric composition algebra) with dimension 8. We will see that they have a close connection with each other.
 
 Let $F$ be a field, and suppose char($F)\neq 2$. In this and the following sections, by an $F$-algebra $A$ we mean (unless further specified) a finite dimensional vector space over $F$ equipped with an $F$-bilinear multiplication. The multiplication here is not necessarily associative, and we do not assume that $A$ has an identity element. Consider an algebra $A$ over $F$ with a quadratic form $q: A\rightarrow F$. We always assume that $q$ is nonsingular, i.e., the bilinear form $\bb \ ,\ \pp$ associated to $q$:
 \[
 \bb x,y\pp:=q(x+y)-q(x)-q(y),
 \]
 for all $x,y\in A$ has radical $\{0\}$.
 
\begin{defi}\label{def 21}
A composition algebra $A$ over a field $F$ with multiplication $x\cdot y$ is an algebra with a nonsingular quadratic form $q$ on $A$ satisfying:
\[
q(x\cdot y)=q(x)q(y).
\]
Furthermore, if $A$ has an identity element $e$, we call $A$ a unital composition algebra.
\end{defi}
	
Every element of a unital composition algebras satisfies a quadratic polynomial. This is the minimal polynomial if the element is not a scalar multiple of the identity. More precisely, we have:

\begin{prop}\label{prop 22}
Every element $x$ of a unital composition algebra $A$ satisfies 
\[
x\cdot x-\bb x,e\pp x+q(x)e=0.
\]
For all $x,y\in A$. By linearizing the above equation, we get:
\[
x\cdot y+y\cdot x-\bb x,e\pp y-\bb y,e\pp x+\bb x,y\pp e=0.
\]
\end{prop}  	

\begin{proof}
See \cite[Proposition 1.2.3]{Sp}.
\end{proof}
	
We now introduce conjugation in $A$. Consider the mapping
$~\bar{}: A\rightarrow A$ given by
\[
\bar{x}=\bb x,e\pp e-x,
\]		
for all $x\in A$. We call $\bar{x}$ the conjugate of $x$. It is easy to see that $~\bar{}~$ is a linear map with $\overline{x\cdot y}=\bar{y}\cdot\bar{x}$, and $\bar{\bar{x}}=x$ for all $x,y\in A$. So $~\bar{}~$ is an involution on $A$.

The following lemmas hold in every unital composition algebras and can be directly deduced by linearizing the quadratic form $q$ and using composition condition (see \cite[\S 1.2, 1.3]{Sp} for he proof). We will use them later:

\begin{lm}\label{lm 23}
We have
\begin{itemize}
	\item [(1)] $\bb  x\cdot z, y\cdot z\pp =\bb x,y\pp q(z),\quad \bb  z\cdot x, z\cdot y\pp =q(z)\bb  x,y\pp$,
	\item [(2)] $\bb  x\cdot z,y\cdot w\pp +\bb  x\cdot w, y\cdot z\pp =\bb  x,y\pp \bb  z,w\pp $,
\end{itemize}
for all $x,y,z\in A$.
\end{lm}

\begin{lm}\label{lm 24}
We have
\begin{itemize}
	\item [(1)] $x\cdot (\bar{x}\cdot y)=q(x)y, \quad (x\cdot \bar{y})\cdot y=q(y)x$,
	\item [(2)] $x\cdot(\bar{y}\cdot z)+y\cdot(\bar{x}\cdot z)=\bb  x,y\pp z$,
	\item [(3)] $(x\cdot\bar{y})\cdot z+(x\cdot \bar{z})\cdot y=\bb  y,z\pp x$,
\end{itemize}
for all $x,y,z\in A$.
\end{lm}

Given a unital composition algebra $A$ with dimension $n$, one can get a new composition algebra with dimension $2n$ by using Cayley-Dickson process (see \cite[\S 33.C]{KMRT}). We now have the well-know classification of unital composition algebras:

\begin{thm}\label{thm 25}
	Every unital composition algebra over $F$ is obtained by the Cayley-Dickson process. The possible dimensions are $1,2,4$ and $8$. Composition algebras of dimension $1$ or $2$ are commutative and associative, those of dimension $4$ are associative but not commutative, and those of dimension $8$ are neither commutative nor associative.
\end{thm}

\begin{proof}
	See \cite[Theorem 33.17]{KMRT}.
\end{proof}

We call $A$ a quadratic \etale algebra if $\dim A=2$, a quaternion algebra (resp. an octonion algebra) if $\dim A=4$ (resp. $\dim A=8$). The Cayley-Dickson process applied to an octonion algebra does not yield a composition algebra, so octonion algebras have highest dimension. We will use these algebras to construct the exceptional group $G_2$.

 Recall that an element $x\in A$ is called isotropic if $q(x)=0$. The quadratic form $q$ is called isotropic if there exist nonzero isotropic elements in $A$. From \cite[Theorem 33.19]{KMRT}, there is only one isomorphism class of octonion algebras with isotropic quadratic forms. We call $C_s$ the split octonion algebra if they have isotropic quadratic form. The split octonion algebra $C_s$ can be constructed from the split quaternion algebra, but we have a remarkable way to express its multiplication rule (see \cite[\S 2]{Sp2} and \cite[\S 5]{GY1}): 
 
 Consider a free $\ZZ$-module $W$ of rank 3, and let $W^{\vee}=\Hom_{\ZZ}(W,\ZZ)$ be the dual module of $W$. By fixing a basis in $W$, we set 
 \[
 W=\ZZ u_1+\ZZ u_2+\ZZ u_3,\quad W^\vee=\ZZ v_1+\ZZ v_2+\ZZ v_3,
 \]
 where $v_i(u_j)=\delta_{ij}$ for $i,j\in\{1,2,3\}$. We also fix an isomorphism $\wedge^3 W$ with $\ZZ\simeq \ZZ\cdot (u_1\wedge u_2\wedge u_3)$. Then we get an isomorphism:
 \[
 \wedge^2 W\simeq W^\vee \quad \text{given by $w_1\wedge w_2\mapsto (w_1\wedge w_2, -)$},
 \]
where $(w_1\wedge w_2,w):=w_1\wedge w_2\wedge w$ for any $w\in W$. Similarly, we also have an isomorphism $\wedge^2 W^\vee\simeq W$ by duality. For instance, $u_1\wedge u_2$ is identified with $v_3$, $v_1\wedge v_3$ is identified with $-u_2$, etc.
 
 Let $\Lambda$ be the space of matrices of the form:
 \[
 \left(\begin{array}{cc}
 	a&w\\
 	\phi&b
 \end{array}
 \right),\quad\text{where $a,b\in\ZZ, w\in W$, and $\phi\in W^\vee$}.
 \]
 This is a free $\ZZ$-module of rank 8. If we define addition by matrices addition, and multiplication by:
 \[
 \left(\begin{array}{cc}
 	a_1 &w_1\\
 	\phi_1&b_1
 \end{array}
 \right)\cdot
 \left(\begin{array}{cc}
 	a_2&w_2\\
 	\phi_2&b_2
 \end{array}
 \right)=
 \left(\begin{array}{cc}
 	a_1a_2-\phi_2(w_1) & a_1w_2+b_2w_1+\phi_1\wedge\phi_2\\
 	a_2\phi_1+b_1\phi_2+w_1\wedge w_2 & b_1b_2-\phi_1(w_2)
 \end{array}
 \right)
 \]
 then $\Lambda$ is an algebra over $\ZZ$. Moreover, consider the quadratic form $q$ on $\Lambda$ given by:
 \[
 q:\left(\begin{array}{cc}
 	a &w\\
 	\phi&b
 \end{array}
 \right)\mapsto
 ab+\phi(w).
 \]
 It is easy to see that $q(x\cdot y)=q(x)q(y)$, so that $\Lambda$ is a unital composition algebra  over $\ZZ$. Here $\Lambda$ has a standard basis $e_1, e_2, \{u_i, v_i\}_{i=1,2,3}$, where 
 \[
 e_1=\left(\begin{array}{cc}
 	1&0\\
 	0&0
 \end{array}
 \right),\quad
 e_2=\left(\begin{array}{cc}
 	0&0\\
 	0&1
 \end{array}
 \right),\quad
 u_i=\left(\begin{array}{cc}
 	0&u_i\\
 	0&0
 \end{array}
 \right),\quad
 v_i=\left(\begin{array}{cc}
 	0&0\\
 	v_i&0
 \end{array}
 \right).
 \]
 The identity element in $\Lambda$ is the identity matrix $e=e_1+e_2$. Therefore, the split octonion algebra $C_s$ can be viewed as $C_s=\Lambda\otimes_{\ZZ}F$.
 
 Although unital composition algebras have lots of good properties, most composition algebras do not admit an identity. In the rest of this section, we will discuss a special class of composition algebras, called symmetric composition algebras.
 
\begin{defi}\label{def 26}
A symmetric composition algebra $(S,\star)$ is a composition algebra with multiplication $x\star y$, satisfying
\[
\bb  x\star y,z\pp =\bb y\star z,x\pp,
\]
for all $x,y,z\in S$.
\end{defi}
 
Similar to Lemma \ref{lm 23}, \ref{lm 24}, the following results hold in every symmetric composition algebra $(S,\star)$ (see \cite[Lemma 34.1]{KMRT}):

\begin{lm}\label{lm 27}
We have
\begin{itemize}
	\item [(1)] $\bb  x\star z, y\star z\pp =\bb  x,y\pp q(z)$,
	\item [(2)] $\bb  z\star x, z\star y\pp =q(z)\bb  x,y\pp \pp $,
	\item [(3)] $\bb  x\star z,y\star w\pp +\bb  x\star w, y\star z\pp =\bb  x,y\pp \bb  z,w\pp $,
\end{itemize}
for all $x,y,z\in S$.
\end{lm}

\begin{lm}\label{lm 28}
We have
\begin{itemize}
	\item [(1)] $(x\star y)\star z+(z\star y)\star x=\bb  x,z\pp y$,
	\item [(2)] $x\star(y\star z)+z\star(y\star x)=\bb  x,z\pp y$,
\end{itemize}
for all $x,y,z\in S$. 

In particular, we have
$(x\star y)\star x=x\star(y\star x)=q(x)y$.
\end{lm}

Note that we have $(x\star y)\star x=x\star(y\star x)=q(x)y$ for all $x,y\in S$ by Lemma \ref{lm 28}. Generally, we can show that an algebra with quadratic form $q$ satisfying this equation is a symmetric composition algebra.
 
It turns out that unital composition algebras have a close relation to symmetric composition algebras. Starting from a unital composition algebra $(A,\cdot)$ over $F$, we can always define a symmetric composition algebra $(A,\star)$ given by:
\[
x\star y=\bar{x}\cdot\bar{y},
\] 
for all $x,y\in A$. Symmetric composition algebras do not have identity, but the identity element $e\in A$ plays a special role in the corresponding symmetric composition algebra $(A,\star)$: it is an idempotent ($e\star e=e$) and satisfies 
\[
e\star x=x\star e=-x, 
\]
for any $x\in A$, with $\bb  x,e\pp =0$. We call an element which satisfies the above condition a para-unit. In particular, symmetric composition algebras corresponding to octonion algebras (resp. split octonion algebras) are called para-octonion algebras (resp. split para-octonion algebras), denoted by $(C,\star)$ (resp. $(C_s,\star)$). Using the same symbols as above ($C_s\simeq \Lambda\otimes_{\ZZ}F$, with basis $e_1,e_2, \{u_i,v_i\}_{i=1,2,3}$), we obtain the multiplication table of the split para-octonion algebra $(C_s,\star)$ as Table \ref{tab 1} below (we write ``$\cdot$" instead of $0$ for clarity).

\begin{table}[H]
				\begin{tabular}{cc|cc|ccc|ccc}
					\multicolumn{9}{c}{y} \\ \hline
					&$\star$&$e_1$&$e_2$&$u_1$&$u_2$&$u_3$&$v_1$&$v_2$&$v_3$\\ \hline
					&$e_1$&$ e_2$
					&$\cdot$&$\cdot$&$\cdot$&$\cdot$
					&$-v_1$&$-v_2$&$-v_3$\\
					&$e_2$&$\cdot$&$e_1$
					&$-u_1$&$-u_2$&$-u_3$
					&$\cdot$&$\cdot$&$\cdot$\\ \hline
					&$u_1$&$-u_1$&$\cdot$
					&$\cdot$&$v_3$&$-v_2$
					&$-e_1$&$\cdot$&$\cdot$\\
					&$u_2$&$-u_2$&$\cdot$
					&$-v_3$&$\cdot$&$v_1$
					&$\cdot$&$-e_1$&$\cdot$\\ 
					$x$&$u_3$&$-u_3$&$\cdot$
					&$v_2$&$-v_1$&$\cdot$
					&$\cdot$&$\cdot$&$-e_1$\\ \hline
					&$v_1$&$\cdot$&$-v_1$
					&$-e_2$&$\cdot$&$\cdot$
					&$\cdot$&$u_3$&$-u_2$\\
					&$v_2$&$\cdot$&$-v_2$
					&$\cdot$&$-e_2$&$\cdot$
					&$-u_3$&$\cdot$&$u_1$\\
					&$v_3$&$\cdot$&$-v_3$
					&$\cdot$&$\cdot$&$-e_2$
					&$u_2$&$-u_1$&$\cdot$\\
				\end{tabular}
				\caption{The split para-octonion algebra multiplication $x\star y$}
				\label{tab 1}
			\end{table}

With respect to this basis, the corresponding bilinear form $\bb\ ,\ \pp$ is given by 
 \[
 \bb e_1,e_2\pp=1,\quad \bb u_i,v_j\pp=\delta_{ij},
 \]
 where all others are equal to $0$. The involution $~\bar{~}$ on $(C_s,\star)$ is given by:
 \[
 \overline{e}_1=e_2,\quad \overline{e}_2=e_1,\quad \overline{u}_i=-u_i,\quad \overline{v}_i=-v_i.
 \] 
 We prefer to use para-octonion algebras $(C,\star)$ instead of octonion algebras $(C,\star)$ in the following sections. One reason is that it is convenient to work with the spin group $\Spin(C,\star)$. We refer Remark \ref{rk 32} for more details.

\begin{rk}\label{rk 29}{\rm
	Recall that a subspace $U$ in an algebra $A$ with a quadratic form $q$ is isotropic if there exists $x\in U$ such that $q(x)=0$, and $U$ is totally isotropic if $q(x)=0$ for all $x\in U$. The maximal totally isotropic subspaces are totally isotropic subspaces with the highest dimension. All maximal totally isotropic subspaces have the same dimension, which is called the Witt index of $q$. The Witt index is at most equal to half dimension of $A$.
	
	The maximal totally isotropic subspaces in the split octonion (para-octonion) algebra $C_s$ over $F$ is described by Blij-Springer in \cite{Sp3}, and Matzri-Vishne translate the classification to arbitrary composition algebras in \cite{MV}. In $C_s$, Every totally isotropic subspace is of the form 
	\[
	x\star C_s,~ \text{or}~ C_s\star x,
	\] 
	where $x$ is an isotropic element. These spaces have dimension $4$. Furthermore, $x\star C_s=y\star C_s$ if and only if $Fx=Fy$ (see \cite[Theorem 3.1, Proposition 3.2]{MV}).
	
	Let $U$ be a totally isotropic subspace in $C_s$. Consider the intersection of maximal isotropic subspaces:
	\[
	\calL(U):=\cap_{x\in U}(C_s\star x),\quad \calR(U):=\cap_{x\in U}(x\star C_s).
	\]
	We have the following diagram:
	\[
	\begin{tikzcd}
& Fx\arrow[dl, "\mathcal{L}"]&\\
C_s\star x \arrow[rr, "\mathcal{R}\circ\mathcal{R}"]& &x\star C_s.\arrow[ul, "\mathcal{L}"]
\end{tikzcd}
	\]
	for any $x$ isotropic element. We call this diagram the geometric triality graph (see\cite[\S 6.2]{MV}).
}\end{rk} 
 
 \section{The Building of $G_2$}\label{Building}
 
 In this section, we briefly recall the theory of Bruhat-Tits building. We are particularly interested in the special case when the group is of type $G_2$. The theory of buildings of reductive groups over local fields are given by Bruhat-Tits in their series of papers \cite{BT1}, \cite{BT2}, \cite{BT3}, \cite{BT4}, \cite{BT5}. In \cite{BT1}, \cite{BT2}, the building of a classical group $G$ is a polysimplicial complex equipped with an action of $G$, and can be given  as a set of graded lattice chains satisfying certain conditions. 
 
 For a simply connected classical group $G$, let $F$ be a discrete valuation field, and $\calO_F$ be the ring of integers with uniformizer $\pi\in\calO_F$. The building $\calB(G)$ can be realized geometrically as a set of graded lattice chains. To each point $x\in \calB(G)$, the subgroup $\scrG_x$ that stabilizes the graded lattice chain $x$ is a smooth, connected, affine group scheme over $\Spec(\calO_F)$, and called the parahoric subgroup of $\calB(G)$ associated to $x$. We will fix a parahoric subgroup of $\calB(G)$ in the next section and use it to define our affine Grassmannians. We refer \cite{BT3}, \cite{BT5} for the theory of buildings of classical groups, and \cite{GY1} for the building of $G_2$.
 
 Let us first construct the exceptional group of type $G_2$. Generally, let $(V,q)$ be a finite dimensional vector space with a nonsingular quadratic form $q$ over a field $F$, with char($F$) different from $2$. Denote by $\bb \ ,\ \pp$ the bilinear form corresponding to $q$. The special orthogonal group $\SO(V,q)$ is the subgroup of the special linear group $\SL(V,q)$ that preserve the form $q$. The universal covering of $\SO(V,q)$ is the spin group $\Spin(V,q)$ defined by:
 \[
 \Spin(V,q)=\{x\in C_0(V,q)^* \mid xVx^{-1}=V,~ \tau(x)x=1\}.
 \]
 Here $C(V,q)$ is the Clifford algebra of $(V,q)$, which is the factor of the tensor algebra $T(V)=\oplus_{n\geq 0}(V\otimes\cdots\otimes V)$ by the ideal $I(q)$ generated by all elements of the form $x\otimes x-q(x)\cdot 1$ for $x\in V$. It is a $\ZZ/2\ZZ$-graded algebra with $C(V,q)=C_0(V,q)\oplus C_1(V,q)$. We call $C_0(V,q)$ the even Clifford algebra, and $C_1(V,q)$ the odd Clifford algebra (see \cite[Ch. IV]{KMRT}). 	There is a  canonical involution of the Clifford algebra $\tau: C(V,q)\rightarrow C(V,q)$ given by $\tau(x_1\cdots x_d)=x_d\cdots x_1$ for $x_1,\cdots, x_d\in V$. Thus, the spin group $\Spin(V,q)$ is a subgroup of $C_0(V,q)$. We have an exact sequence:
 \[
 1\rightarrow \ZZ/2\ZZ\rightarrow \Spin(V,q)\rightarrow \SO(V,q)\rightarrow 1.
 \]
In particular, when consider the split para-octonion algebra $(C_s,\star)$ as an 8-dim vector space over $F$, we can get the corresponding special orthogonal group $\SO(C_s,\star)$ (resp. the spin group $\Spin(C_s,\star)$). The corresponding group scheme $\SSO(C_s,\star)$ (resp. $\SSpin(C_s,\star)$ over $\Spec(F)$ are defined by:
\[
\begin{array}{l}
	\SSO(C_s,\star)(R)=\{g\in \SL(C_{sR},\star)\mid \bb g(x),g(y)\pp=\bb x,y\pp\},\\
	\SSpin(C_s,\star)(R)=\{x\in C_0(C_{sR},\star)^*\mid xC_{sR} x^{-1}=C_{sR}, \tau(x)x=1\},
\end{array}
\]
for any $x,y\in C_{sR}$. Here $R$ is an $F$-algebra, and $C_{sR}:=C_s\otimes_F R$. The following description of $\SSpin(C_s,\star)$ shows that the spin group $\SSpin(C_s,\star)$ is a subgroup of the triple orthogonal group $\SSO_8^{\times 3}$. 
 
 \begin{thm}\label{thm 31}
There is an isomorphism 
\[
\SSpin(C_s,\star)\simeq \RRT(C_s,\star),
\]
where $\RRT(C_s,\star)$ is the related triple group scheme that represents the functor from $F$-algebras to groups that sends R to
\[
\RRT(C_s,\star)(R):=\{(g_1,g_2,g_3)\in \SSO(C_s,\star)(R)^{\times 3} \mid g_i(x\star y)=g_{i+1}(x)\star g_{i+2}(y)\}.
\]
for any $x,y\in C_s\otimes_F R$, $i=1,2,3 \mod 3$.
\end{thm}
	 
See the proof in \cite[Proposition 35.8]{KMRT}. The triple $(g_1,g_2,g_3)\in \SSO(C_s,\star)(R)^{\times 3}$ that satisfies $g_i(x\star y)=g_{i+1}(x)\star g_{i+2}(y)$ is called a related triple. For any given $g_1\in \SO(C_s,\star)(R)$, there exists $g_2, g_3\in \SO(C_s,\star)(R)$ such that $(g_1,g_2,g_3)$ is a related triple (see \cite[Proposition 35.4]{KMRT}). 

\begin{rk}\label{rk 32}{\rm
	We have a similar result for the spin group of the split octonion algebra $(C_s,\cdot)$: There is an isomorphism 
	\[
	\SSpin(C_s,\cdot)\simeq \{(g_1,g_2,g_3)\in \SSO(C_s,\cdot)(R)^{\times 3} \mid g_i(x\cdot y)=\overline{g_{i+1}(\bar{x})\cdot g_{i+2}(\bar{y})}\}.
	\]
	(See \cite[Remark 3.1]{AG}). It is easy to see that the above isomorphism can be obtained from Theorem \ref{thm 31} by regarding $x\star y=\bar{x}\cdot \bar{y}$. Since it is more clear to view the spin group as the subgroup of $\SSO_8^{\times 3}$ preserving multiplication, we prefer to use para-octonion algebras in the following sections. 
}\end{rk}

From Theorem \ref{thm 31}, there exists an action $\rho$ on $\SSpin(C_s,\star)$ given by 
\[
\rho:(g_1,g_2,g_3)\mapsto (g_2,g_3,g_1).
\]
Further, for any $g\in\SSO(C_s,\star)$, let $\hat{g}$ be the automorphism of $(C_s,\star)$ given by $\hat{g}(x)=\overline{g(\bar{x})}$. Then $\hat{g}$ is also an element in $\SSO(C_s,\star)$, and there is an involution on $\SSpin(C_s,\star)$ given by:
\[
\theta:(g_1,g_2,g_3)\mapsto (\hat{g}_1,\hat{g}_3,\hat{g}_2).
\]
The involution $\theta$ and the $\ZZ/3\ZZ$-action $\rho$ together generate an action $S_3$ on $\SSpin(C_s,\star)$, whose group of fixed points, is precisely the exceptional group $G_2$:

\begin{prop}\label{prop 32}
For the given action of $S_3$ on $\SSpin(C_s,\star)$, we have a simple, simply connected group scheme
\[
G=\SSpin(C_s,\star)^{S_3}\simeq \SSpin(C_s,\star)^{A_3}\simeq \AAut(C_s,\star),
\]
where $\AAut(C_s,\star)$ is the automorphism group scheme
\[
\AAut(C_s,\star)(R):=\{g\in \SSO(C_s,\star)(R) \mid g(x\star y)=g(x)\star g(y) \},
\]
for any $F$-algebra $R$, $x,y\in C_s\otimes_{F} R$. This group $G$ is of type $G_2$.
\end{prop}

\begin{proof}
	See\cite[\S 35]{KMRT}.
\end{proof}

Since there is an action $S_3$ on $\SSpin(C_s,\star)$, one can form the building $\calB(\SSpin(C_s,$ $\star))$ equipped with an action $S_3$. It is natural to ask whether the set of fixed points of $\calB(\SSpin(C_s,\star))$ under $S_3$ is the building of $\calB(G)$. When $p\neq 2,3$, Prasad-Yu in \cite{PY} show that it is indeed true. In the general case, the result is given by Gan-Yu in \cite{GY1}. So we have
\[
\calB(G)\simeq \calB(\SSpin(C_s,\star))^{S_3}.
\] 
Gan-Yu give an explicit description of the building $\calB(G)$, and show that there is a bijection between the building $\calB(G)$ and the set of graded lattice chains satisfying certain conditions.

Recall that a lattice $L$ in $(C_s,\star)$ is a finitely generated projective $\calO_F$-submodule in $C_s$, such that $L\otimes_{\calO_F}F=C_s$. A lattice chain $\{L_i\}_{i\in I}$ is a totally ordered set of lattices of $C_s$, such that
\[
L_{i_0}\subsetneq L_{i_1}\cdots \subsetneq L_{i_{n}}\subsetneq \pi^{-1}L_{i_0},
\]
for $I=\{i_0\leq i_1\leq\cdots\leq i_n\}$. A graded lattice chain is a pair $(\{L_i\}_{i\in I}, c)$ where $\{L_i\}_{i\in I}$ is a lattice chain, and $c$ is a strictly decreasing map from $\{L_i\}_{i\in I}$ to $\RR$ such that
\[
c(\lambda L)=\ord(\lambda)+c(L),
\]
for $\lambda\in F, L\in \{L_i\}_{i\in I}$. If for any lattice $L\in \{L_i\}_{i\in I}$, the dual space $L^{\vee}=\{x\in C_s \mid \bb x,L\pp\subset \calO_F\}$ is also in $\{L_i\}_{i\in I}$, then we call $\{L_i\}_{i\in I}$ a self-dual lattice chain. 

An order in $C_s$ is an $\calO_F$-lattice which is closed under multiplication (i.e., $x\star y\in L$ for all $x,y\in L$), and contains the para-unit $e$. This order is said to be maximal if it is not contained in any larger one. Since we always assume char$(F)\neq 2$ in this paper, we have

\begin{prop}\label {prop 33}
Suppose that $L$ is a lattice in $C_s$, which is closed under multiplication, then $L$ is a maximal order if and only if $L$ is self-dual with respect to $\bb \ ,\ \pp$.	
\end{prop}

\begin{proof}
	See \cite[Proposition 5.1]{GY1}. 
\end{proof}

\begin{rk}\label{rk 34}{\rm
	In \cite{GY1}, Gan-Yu define that an order $L$ in the split octonion algebra $(C_s,\cdot)$ is an $\calO_F$-lattice which is a unital ring. We transfer this definition to the para-octonion algebra $(C_s,\star)$ here. Note that any order is closed under conjugation, so that the condition $x\cdot y\in L$ is equivalent to 
	\[
	x\star y=\bar{x}\cdot\bar{y}\in L
	\]
	for all $x,y\in L$. Meanwhile, the unit $e$ in $(C_s,\cdot)$ becomes the para-unit $e$ in $(C_s,\star)$. Hence, we define that an order in a para-octonion algebra is an $\calO_F$-lattice which is closed under multiplication, and contains the para-unit $e$ as above.
}\end{rk}

We are now ready to describe vertices in the building $\calB(G)$ when $G=\AAut(C_s,\star)$. By fixing a maximal split torus $S\subset G$, we can identify the apartment $\mathcal{A}(S)$ with the real vector space $X_+(S)\otimes \RR$. Here $X_+(S)$ is the cocharacter group. Having fixed the origin, Gan-Yu show that the closed chamber in $\mathcal{A}(S)$ is a triangle with three vertices. Each vertex is corresponding to a type of lattices (or lattice chains) in $\calB(G)$.

\begin{thm}\label{thm 35}
There are three types of vertices in building $\calB(G)$. The vertices of type $i$ are in natural bijection with the orders satisfying certain conditions as follows:
\begin{itemize}
	\item type $1$: $L$ is a maximal order in $C_s$. 
	The corresponding graded lattice chain is a pair $(\{\pi^i L\},c)$:
	\[
	\cdots\subsetneq\pi L\subsetneq L\subsetneq \pi^{-1}L\subsetneq\cdots, 
	\]
	where $c(L)=0$.
	\item type $2$: $L$ is an order satisfying $L\subsetneq L^{\vee}\subsetneq \pi^{-1}L,~ L^{\vee}\star L^{\vee}\subset \pi^{-1}L$. The corresponding graded lattice chain is the following:
	\[
	\cdots\subsetneq \pi L^\vee \subsetneq L\subsetneq L^{\vee}\subsetneq \pi^{-1}L \subsetneq\cdots, 
	\]
	where $c(L)=0, ~c(L^\vee)=-\frac12$.
	\item type $3$: $L$ is an order satisfying $L\subsetneq L^{\vee}\subsetneq \pi^{-1}L$, with $M:=\pi L^{\vee}\star L^{\vee}+L$ is self-dual. The corresponding graded lattice chain is the following:
	\[
	\cdots\subsetneq \pi L^\vee\subsetneq L\subsetneq M \subsetneq L^{\vee}\subsetneq \pi^{-1}L \subsetneq\cdots, 
	\]
	where $c(L)=0, ~c(M)=-\frac13,~c(L^\vee)=-\frac23$.
 \end{itemize}
\end{thm}

\begin{proof}
	See \cite[\S 9]{GY1}.
\end{proof}

\begin{ex}\label{ex 36}{\rm 

Let $(C_s,\star)=\Lambda\otimes_{\ZZ}F$ be the split para-octonion algebra, with the basis $e_1,e_2, \{u_i, v_i\}_{i=1,2,3}$ as in Table \ref{tab 1}. 

(1). The trivial example is the lattice 
\[
\LL_1=\calO_F\bb e_1, e_2, u_1, u_2, u_3, v_1, v_2, v_3\pp.
\] 
It is easy to see that $\LL_1$ is self-dual, and closed under multiplication table, so $\LL_1$ is a maximal order by Proposition \ref{prop 33}. The corresponding vertex in $\calB(G)$ is of type 1. We call $\LL_1$ the standard lattice of type 1.

(2). Let 
\[
\LL_2=\calO_F\bb e_1, e_2, \pi u_1, u_2, u_3, v_1, \pi v_2, v_3\pp.
\]	
We see that $\LL_2$ is closed under multiplication.  The dual lattice is
\[
\LL_2^\vee=\calO_F\bb e_1, e_2, u_1, \pi^{-1}u_2, u_3, \pi^{-1}v_1, v_2, v_3\pp.
\]
It is easy to see that $\LL_2\subsetneq \LL_2^{\vee}\subsetneq \pi^{-1}\LL_2$, and 
\[
\LL_2^\vee\star \LL_2^\vee=\calO_F\bb \pi^{-1}e_1, \pi^{-1}e_2, u_1, \pi^{-1}u_2, \pi^{-1}u_3, \pi^{-1}v_1, v_2, \pi^{-1}v_3\pp,
\]
so that $\LL_2^\vee\star \LL_2^\vee=\pi^{-1}\LL_2$. Hence the corresponding lattice chain of $\LL_2$ is a vertex in $\calB(G)$. We call $\LL_2$ the standard lattice of type 2.

(3). Let 
\[
\LL_3=\calO_F\bb e_1, e_2, \pi u_1, \pi u_2, u_3, v_1,  v_2, \pi v_3\pp.
\]
We leave readers to check that $\LL_3$ is closed under multiplication, with $\LL_3\subsetneq \LL_3^{\vee}\subsetneq \pi^{-1}\LL_3$, and
\[
M=\pi\LL_3^\vee\star\LL_3^\vee+\LL_3= \calO_F\bb e_1, e_2, u_1, u_2, \pi^{-1}u_3, v_1,  v_2, \pi v_3\pp
\] 
is self-dual. We call $\LL_3$ the standard lattice of type 3.
}\end{ex}

To each vertex $x$ of type $i$ in $\calB(G)$, the stabilizer of the corresponding lattice (lattice chains) $L_x$ in $G(F)$ is a maximal parahoric subgroup $\scrG_x$. From \cite[Theorem 7.2]{GY1}, the hyperspecial points in $\calB(G)$ corresponding to vertices of type 1. Now we can define the parahoric subgroup $\scrG$ corresponding to the hyperspecial points: Recall that $\Lambda$ is the free $\ZZ$-module of rank $8$, with $C_s=\Lambda\otimes_{\ZZ}F$. Consider $\LL_1=\Lambda\otimes_{\ZZ}\calO_F$  as a submodule in $C_s$. We denote by $\SSO(\Lambda)$ the subscheme of points in the isomorphism group scheme $\Isom(\Lambda)$ that preserve the form and whose determinant is equal to 1.

\begin{defi}\label{def 37}
	The parahoric hyperspecial subgroup corresponding to hyperspecial points $\scrG$ is an affine group scheme over $\Spec(\calO_F)$  that represents the functor from $\calO_F$-algebras to groups that sends R to
	\[
	\scrG(R)=\{g\in \SSO(\Lambda)(R) \mid g(x\star y)=g(x)\star g(y) \},
	\]
	for all $x,y\in \LL_1\otimes_{\calO_F}R$.
\end{defi}

 Similarly, we can use $\LL_2$ (resp. $\LL_3$) in Example \ref{ex 36} to define the parahoric subgroup $\scrG_2$ (resp. $\scrG_3$) (see Definition \ref{def 63}).


\section{Affine Grassmannians for $\scrG$}\label{AffGrass1}

\subsection{Definition}\label{Definition} In this subsection, we will define affine Grassmannians by loop groups. Affine Grassmannians are represented by an ind-scheme. We refer \cite{BL} \cite{BD} for important results on the structure of loop groups and associated affine Grassmannians.

Let $k$ be a perfect field with char$(k)\neq 2$, and let $G_0$ be an algebraic group over $\Spec(k)$. The functor $LG_0$ is from the category of $k$-algebras to sets that sends $R$ to
\[
LG_0(R)=G_0(R\llp t\rlp).
\]
This functor is represented by an ind-scheme, called the loop group associated to $G_0$. We also consider the positive loop group $L^+G_0$, which is the functor on the category of k-algebras, given by
\[
L^+G_0(R)=G_0(R\lp t\rp).
\]
Then $L^+G_0 \subset LG_0$ is a subgroup functor, and the fpqc-quotient $\Gr_{G_0}=LG_0/L^+G_0$ is by definition the affine Grassmannian associated to $G_0$. The fpqc-sheaf $\Gr_{G_0}$ is also represented by an ind-scheme.

In \cite{PR3}, Pappas-Rapoport develop this definition for ``twisted" loop groups as follows: Consider $G$ an algebraic group over $\Spec(k\llp t\rlp)$. The algebraic loop group $LG$ is the functor form the category of $k$-algebras to the sets given by
\[
LG(R)=G(R\llp t\rlp).
\]
Since $R\llp t\rlp$ is a $k\llp t\rlp$-algebra, this definition makes sense. Note that $R\lp t\rp$ is not a $k\llp t\rlp$-algebra, so we cannot use $G$ to define positive loop groups now. By choosing a parahoric subgroup $\scrG$ of $G$, we denote by $L^+\scrG$ the positive loop group that represents the functor form $k$-algebras to sets that sends $R$ to
\[
L^+\scrG(R)=\scrG(R\lp t\rp).
\]
Note that when $G =G_0\otimes_k k\llp t\rlp$, we recover the previous definition in the untwisted case. Since the generic fiber $\scrG_\eta$ of $\scrG$ is always equal to $G$, our affine Grassmannian $\Gr_\scrG$ associated to $\scrG$ is the fpqc-sheaf $L\scrG_\eta/L^+\scrG=LG/L^+\scrG$. We refer to \cite{PR3} for more details. 

A classical prototype is the affine Grassmannian for the  general linear group $\GL_n$. This $\Gr_{\GL_n}$ can be viewed as a functor from $k$-algebras to sets that sends $R$ to
\[
\{L \mid \text{$L$ is a lattice in $R\llp t\rlp^n$, such that $L[t^{-1}]\simeq R\llp t\rlp^n$}\}.
\]

In the rest of this section, let $k$ be a perfect field with char$(k)\neq 2$. The field $F=k\llp t\rlp$ is a discrete valuation field with the ring of integers $\calO_F=k\lp t\rp$, and uniformizer $t\in k\lp t\rp$. Let $(C_s,\star)=\Lambda\otimes_{\ZZ} k\llp t\rlp$ be the split para-octonion algebra over $F$, and $G=\AAut(C_s,\star)$ be the automorphism group scheme of split para-octonion algebra, which is a group of type $G_2$. Consider the hyperspecial parahoric subgroup $\scrG$ corresponding to hyperspecial points in $\calB(G)$ as defined in Definition \ref{def 37}. Then the quotient fpqc-sheaf 
\[
\Gr_{\scrG}=LG/L^+\scrG
\] 
is by definition the affine Grassmannian for  $G_2$.

Our goal in this section is to give an explicit description of this affine Grassmannian $\Gr_{\scrG}$. Inspired by the results from \cite{GY1}, we already know the $k$-points in $\Gr_{\scrG}$: since $L^+\scrG(k)=\scrG(k\lp t\rp)$ is the stabilizer of the standard lattice $\LL_1$, and $LG(k)=G(k\llp t\rlp)$ acts transitively on vertices of type 1, there is a bijection between $k$-points of $\Gr_{\scrG}$ and the set of maximal orders in $C_s$. Our first main theorem is the following:

\begin{thm}\label{thm 41}
	    There is an $LG$-equivariant isomorphism
		\[
		\Gr_{\scrG}\simeq \mathscr{F}
		\]	
		where the functor $\mathscr{F}$ sends a $k$-algebra $R$ to the set of  $R\lp t\rp $-lattices $L$ of $C_s\otimes_{k\llp t\rlp} R\llp t\rlp$ such that
		\begin{itemize}
			\item [(1)] $L$ is self dual under the bilinear form $\bb\ ,\ \pp$, i.e., $L\simeq \Hom_{R\lp t\rp}(L, R\lp t\rp )$.
			\item [(2)] $L$ is an order, i.e., $e\in L$, $L\star L\subset L$.
	    \end{itemize}
\end{thm}

This theorem gives a bijection between $R$-points in the affine Grassmannian for  $\scrG$ and a certain set of $R\lp t\rp$-lattices in $C_s\otimes_{k\llp t\rlp} R\llp t\rlp$ that are self-dual and closed under multiplication.

	
\subsection{Proof of Theorem \ref{thm 41}}\label{ProofofTheorem}

In \cite{PR3}, Pappas-Rapoport show that there is a bijection between the affine Grassmannians for unitary groups and the space of self-dual lattice chains for a hermitian form. Recall that $\LL_1=k\lp t\rp \bb e_1, e_2, u_i, v_i\pp_{i=1,2,3}$ is the standard lattice of type 1 in the split para-octonion algebra $C_s=\Lambda\otimes_{\ZZ}k\llp t\rlp$. Similar to the proof of \cite[Theorem 4.1]{PR3}, it suffices to check the following two parts:
	\begin{itemize}
		\item [(i)] For any $R$, the subgroup of $LG(R)$ consisting of elements that stabilize the standard lattice $\LL_1\otimes_{k\lp t\rp}R\lp t\rp$ agrees with $L^+\scrG(R)$. 
		\item [(ii)] Let $L\in \mathscr{F}(R)$. Then locally for the \etale topology on $R$, there exists $g\in  L\scrG_{\eta}(R)$ such that $L=g(\LL_1\otimes_{k\lp t\rp}R\lp t\rp)$. 
	\end{itemize}

It is easy to see part (i). By Definition \ref{def 37}, the $R\lp t\rp$-points of $\scrG$ is the stabilizer of the base change of the standard lattice $\LL_1\otimes_{k\lp t\rp}R\lp t\rp$ (we also call $\LL_1\otimes_{k\lp t\rp}R\lp t\rp$ the standard lattice in this subsection if there is no confusion). We will focus to prove part (ii). 

Observe that there is a natural morphism $LG\rightarrow \scrF$, given by $g\mapsto g(\LL_1\otimes R\lp t\rp)$. Indeed, set $L=g(\LL_1\otimes R\lp t\rp)$. We have 
\[
L\star L=g(\LL_1\otimes R\lp t\rp)\star g(\LL_1\otimes R\lp t\rp)=g((\LL_1\otimes R\lp t\rp)\star (\LL_1\otimes R\lp t\rp))\subset g(\LL_1\otimes R\lp t\rp)
\]
since $g$ preserves the multiplication and $\LL_1$ is closed under multiplication. Also $L$ is self-dual since $g$ preserves the form $\bb \ ,\ \pp$. Thus, $L\in \scrF(R)$. By part (i), we see that the stabilizer of the standard lattice $\LL_1\otimes R\lp t\rp$ is $L^+\scrG(R)$, and so we obtain
\[
LG/L^+\scrG\rightarrow \scrF.
\]
It is enough to show part (ii) for $R$ local strictly henselian.


In the rest of this subsection, we assume that $(R,\calM)$ is a local strictly henselian ring with maximal ideal $\calM$. For any $L\in\scrF(R)$, by Proposition \ref{prop 33}, $L$ is a maximal order in $C_s\otimes_{k\llp t\rlp}R\llp t\rlp$, which contains a para-unit element $e\in L$. Since $L$ is self-dual and the residue field of $R$ is separably closed, there exists a basis $\{x_i\}_{i=1}^8$ of $L$ such that $\bb x_i,x_{9-j}\pp=\delta_{ij}$. Consider $e=\sum a_i x_i$ for all $a_i\in R\lp t\rp$. We set:
\[
e_1=\sum_{i=1}^4 a_i x_i,\quad e_2=\sum_{i=5}^8 a_i x_i.
\]
Then $q(e_1)=q(e_2)=0$, and $\bb e_1,e_2\pp=1$ by $q(e)=1$. Moreover, consider the minimal polynomial for $e_1$ in the split octonion algebra $C_s\otimes_{k\llp t\rlp}R\llp t\rlp$:
\[
e_1\cdot e_1-\bb e_1,e\pp e_1+q(e_1)e=0. 
\]
We get $e_1\cdot e_1=e_1$ by $\bb e_1,e\pp=1$ and $q(e_1)=0$. Thus in the split para-octonion algebra $(C_s\otimes_{k\llp t\rlp}R\llp t\rlp,\star)$, we have
\[
e_2\star e_2=\bar{e}_2\cdot\bar{e}_2=e_1,
\]
by $\bar{e}_2=\bb e_2,e\pp e-e_2=e_1$. Similarly, we have $e_1\star e_1=e_2$. For $e_1\star e_2$, it is easy to see that
\begin{flalign*}
			e_1\star e_2&=\bar{e}_1\cdot\bar{e}_2=e_2\cdot e_1\\
			&=(e-e_1)\cdot e_1=0.
\end{flalign*} 
Also we get $e_2\star e_1=0$ by similar calculation. Therefore, the sublattice $R\lp t\rp e_1\oplus R\lp t\rp e_2$ is a hyperbolic space satisfying:
\begin{equation}\label{eq42}
\begin{array}{l}
		e_1\star e_1=e_2, \quad e_2\star e_2=e_1,\\
		e_1\star e_2=e_2\star e_1=0,\\
		q(e_1)=q(e_2)=0, \quad \bb  e_1,e_2\pp =1.	
\end{array}	
\end{equation}

Readers may notice that the above $e_1, e_2$ play the same roles as for $e_1, e_2$ in the Table \ref{tab 1} (that is the reason we keep the same symbol). We claim that $e_1, e_2$ are primitive elements in $L$ (elements that extends to a basis of $L$). Since $(R,\calM)$ is a local ring, the canonical quotient map:
\[
R\lp t\rp\rightarrow R\rightarrow \kappa=R/\calM
\]
gives a base change of $L\rightarrow L\otimes_{R\lp t\rp}\kappa$. The image of $e_1, e_2$ in $L\otimes_{R\lp t\rp}\kappa$ are non-zero elements by $\bb e_1,e_2\pp=1$. By Nakayama's lemma, we can extend $\{e_1, e_2\}$ as a basis of $L$. So they are primitive elements. Choose $y_3,...,y_8\in L$ such that $\{e_1, e_2, y_3, ... , y_8\}$ is a basis of $L$. Consider the sublattice $L_0\subset L$ that orthogonal to $\{e_1, e_2\}$,
\[
L_0:=\{x\in L \mid \bb  x,e_1\pp =0,~ \bb  x, e_2\pp =0\}.
\]
For any $x\in L$, we can write $x$ as
$
x=b_1 e_1+b_2 e_2+\sum_{i=3}^8 b_i y_i,
$
for $b_i\in R\lp t\rp$. By replacing $y_i$ by 
$
y_i'=y_i-\bb y_i, e_2\pp e_1-\bb y_i, e_1\pp e_2,
$
we get $\bb y_i',e_1\pp=\bb y_i', e_2\pp=0$. Thus, replacing $y_i$ by $y_i'$ if necessary, we can always assume that $y_i\in L_0$. Hence $L=R\lp t\rp e_1\oplus R\lp t\rp e_2\oplus L_0$, and $L_0$ is a sublattice of rank 6. Note that for any $x\in L_0$, we have $\bb x,e\pp=\bb x,e_1+e_2\pp=0$. Thus $x\star e=e\star x=-x$ since $e$ is the para-unit element. Set $L_i=L_0\star e_i$ for $i=1,2$. By Remark \ref{rk 29}, $L_i$ is a totally isotropic subspace with rank $\leq 3$.

\begin{lm}\label{lm 51} We have

(1) $L_i\subset L_0$ for $i=1,2$.
 
(2) $e_i\star L_i=0,~ L_i\star e_{i+1}=0$ for $i=1,2\mod 2$.
\end{lm} 
 
\begin{proof}
	(1). For any $x\in L_1$, we can write $x$ as $x_1\star e_1$ for $x_1\in L_0$. By Definition \ref{def 26}, we obtain
	\[
	\begin{array}{l}
	\bb x_1\star e_1, e_1\pp=\bb e_1\star e_1,x_1\pp=\bb e_2,x_1\pp=0,\\
	\bb x_1\star e_1, e_2\pp=\bb e_1\star e_2, x_1\pp=0.	
	\end{array}
	\]
	So that $x\in L_0$, and $L_1\subset L_0$. We have a similar result for $L_2$.
	
	(2). Again, we only consider $L_1$ case. For any $x=x_1\star e_1\in L_1$, we have $e_1\star(x_1\star e_1)=q(e_1)x_1=0$ by Lemma \ref{lm 28}. Hence $e_1\star L_1=0$. For $(x_1\star e_1)\star e_2$, consider
	\[
	(x_1\star e_1)\star e_2+(e_2\star e_1)\star x_1=\bb x_1,e_2\pp e_1.
	\]
	Since $e_2\star e_1=0$, and $\bb x_1,e_2\pp=0$, we get $x\star e_2=(x_1\star e_1)\star e_2=0$ for all $x\in L_1$. The proof for $L_2$ case is similar.
\end{proof} 
 
For any $x\in L_0$, we obtain
\begin{flalign*}
		   x&=\bb e_1.e_2\pp x\\
			&=(e_1\star x)\star e_2+(e_2\star x)\star e_1.
\end{flalign*} 
by Lemma \ref{lm 28}. Since $e_1\star x, e_2\star x\in L_0$, we get $L_0=L_1+L_2$. Each $L_i$ is totally isotropic with rank $\leq 3$, so that $L_0=L_1\oplus L_2$. Since $L_1, L_2$ are totally isotropic, and $\bb \ ,\ \pp$ restricted to $L_0$ is nondegenerate, we see that $L_1, L_2$ are in duality, i.e., the isomorphism $L_1\rightarrow L_2^\vee$ given by $x\mapsto \bb x,-\pp$ induce $L_1\simeq \Hom_{R\lp t\rp}(L_2,R\lp t\rp)$. From that, we can set the basis for $L_i$:
\[
L_1=R\lp t\rp\bb u_1,u_2,u_3\pp,\quad L_2=R\lp t\rp\bb v_1,v_2,v_3\pp.
\]
with $\bb u_i, v_i\pp=\delta_{ij}$. 

\begin{lm}\label{lm 52}
We have 
\[
\begin{array}{c}
u_i\star e_1=e_2\star u_i=-u_i,\\
v_i\star e_2=e_1\star v_i=-v_i,	
\end{array}
\]
for $i=1,2,3$.
\end{lm} 
 
\begin{proof}
We can directly get these equations from Lemma \ref{lm 51} (2): Since for any $x\in L_0$, we have $e\star x=x\star e=-x$. Thus $u_i\star e=u_i\star (e_1+e_2)=-u_i$. By $u_i\star e_2=0$, we obtain $u_i\star e_1=-u_i$. Similarly we can get the rest equations.
\end{proof} 

So far, we find the basis $e_1,e_2, \{u_i,v_i\}_{i=1,2,3}$ of $L$. We claim that the multiplication table of this basis is the same as Table \ref{tab 1}. It remains to check $u_i\star v_j$ (resp. $v_i\star u_j$), and $u_i\star u_j$ (resp. $v_i\star v_j$) for $i=1,2,3$. Let us check $u_i\star v_j$ first. By Lemma \ref{lm 52}, and Lemma \ref{lm 28}, we deduce that
\begin{flalign*}
u_i\star v_j &=-(u_i\star e_1)\star v_j
              =	(v_j\star e_1)\star u_i-\bb u_i,v_j\pp e_1\\
             &= -\delta_{ij}e_1.
\end{flalign*}
Similarly, we also get 
$
v_i\star u_j=-\delta_{ij}e_2
$
for $i=1,2,3$. 

Next, for $u_i\star u_j$ (resp. $v_i\star v_j$), we claim that $u_i\star u_j\in L_2$. It is easy to see that $\bb u_i\star u_j,e_1\pp=\bb u_j\star e_1, u_i\pp=-\bb u_j,u_i\pp$. Here $\bb u_j,u_i\pp=0$ since $L_i$ is totally isotropic. Also from $u_j\star e_2=0$, we get $\bb u_i\star u_j,e_2\pp=\bb u_j\star e_2, u_i\pp=0$. Thus, $u_i\star u_j\in L_0$. We can also get $v_i\star v_j\in L_0$ by the same way. Moreover, by Lemma \ref{lm 28}, we obtain
\begin{flalign*}
	(u_j\star u_i)\star e_2+(e_2\star u_i)\star u_j=\bb u_j,e_2\pp u_i=0,
\end{flalign*}
which implies
\[
	u_i\star u_j=-(e_2\star u_i)\star u_j=(u_j\star u_i)\star e_2,
\]
and so $u_i\star u_j\in L_0\star e_2=L_2$. Take $u_1\star u_2$ for instance. Since
\[
\begin{array}{l}
\bb u_1\star u_2, u_1\pp=-\bb u_1\star u_2,u_1\star e_1\pp=-q(u_1)\bb u_2,e_1\pp,\\
\bb u_1\star u_2, u_2\pp=-\bb u_1\star u_2,e_2\star u_2\pp=-q(u_2)\bb u_1,e_2 \pp,	
\end{array}
\] 
by Lemma \ref{lm 27}, we get $\bb u_1\star u_2, u_1\pp=\bb u_1\star u_2, u_2\pp=0$ as $q(u_1)=q(u_2)=0$. From that, we can assume 
\begin{equation}\label{eq45}
	u_1\star u_2=\lambda v_3
\end{equation}
for some $\lambda\in R\lp t\rp$. Multiplying $v_1$ on the right side of equation (\ref{eq45}), we get $(u_1\star u_2)\star v_1=\lambda(v_3\star v_1)$. Since $(u_1\star u_2)\star v_1+(v_1\star u_2)\star u_1=\bb u_1,v_1\pp u_2$ by Lemma \ref{lm 28}, and $v_1\star u_2=0$, $\bb u_1,v_1\pp=1$ from the above, we have $(u_1\star u_2)\star v_1=u_2$. Hence $v_3\star v_1= \lambda^{-1} u_2$. Therefore, $\lambda, \lambda^{-1}\in R\lp t\rp$ by $L\star L\subset L$, which implies $\lambda$ is a unit in $R\lp t\rp$. We can assume $\lambda=1$, so that 
\[
u_1\star u_2=v_3. 
\]
We perform similar calculations for other $u_i*u_j$ and $v_i*v_j$. The multiplication tables are the following:\\
	
	\begin{minipage}{\textwidth}
		\begin{minipage}[t]{0.45\textwidth}
			\centering
			\makeatletter\def\@captype{table}\makeatother\caption{$u_i*u_j$}\vspace{-3mm}
			\begin{tabular}{|c|c|c|c|}\hline
				$*$& $u_1$&$u_2$&$u_3$\\
				\hline
				$u_1$& $0$& $v_3$ & $-v_2$\\
				\hline
				$u_2$& $-v_3$& 0&$v_1$\\
				\hline
				$u_3$& $v_2$& $-v_1$&0\\
				\hline
			\end{tabular}
		\end{minipage}
		\begin{minipage}[t]{0.45\textwidth}
			\centering
			\makeatletter\def\@captype{table}\makeatother\caption{$v_i*v_j$}\vspace{-3mm}
			\begin{tabular}{|c|c|c|c|}\hline
				$*$& $v_1$&$v_2$&$v_3$\\
				\hline
				$v_1$& 0& $u_3$ & $-u_2$\\
				\hline
				$v_2$& $-u_3$& 0&$u_1$\\
				\hline
				$v_3$& $u_2$& $-u_1$&0\\
				\hline
			\end{tabular}
		\end{minipage}
	\end{minipage}\\

From the above, we finish the multiplication table of the basis $e_1, e_2, \{u_i, v_i\}_{i=1,2,3}$. It is easy to see that this multiplication table is the same as Tabel \ref{tab 1} (which is the multiplication table of the standard lattice $\LL_1\otimes R\lp t\rp$). Thus, there exists an element $g\in \SL_8(R\llp t\rlp)$ such that $g(\LL_1\otimes R\lp t\rp)=L$, with $g(x\star y)=g(x)\star g(y)$. By Lemma \ref{lm 28}, the multiplication $\star$ determines the quadratic form $q$, so that $g$ also preserves the form $q$. Therefore, $g\in G(R\llp t\rlp)=LG(R)$. That proves Theorem \ref{thm 41}.


\section{Affine Grassmannians for Other Parahoric Subgroups}\label{AffGrass2}
	
So far, we described the affine Grassmannian for $\scrG$, where $\scrG$ is the maximal parahoric subgroups corresponding to a hyperspecial point in the building $\calB(G)$. We can view $\scrG$ as the stabilizer subgroups of the standard lattice $\LL_1$ (type 1). According to Theorem \ref{thm 35}, there are two other maximal parahoric subgroups up to conjugation, corresponding to vertices of type 2 and type 3 in the building $\calB(G)$. It is natural to ask if we can describe the affine Grassmannians for these parahoric subgroups.

To answer this question, first we need to fix parahoric subgroups that we want. The notation here is the same as in \S 5. That is: $k$ is a perfect field with char$(k)\neq 2$, $F=k\llp t\rlp$ is a discrete valuation field with the ring of integers $\calO_F = k\lp t\rp$, and $(C_s, \star)=\Lambda\otimes_\ZZ k\llp t\rlp$ is the split para-octonion algebra over $F$. Consider the lattice $\LL_2, \LL_3 \subset C_s$ in Example \ref{ex 36}:
\[
\LL_2=k\lp t\rp\bb e_1,e_2,\pi u_1, u_2,u_3, v_1,\pi v_2, v_3\pp,
\] 
and 
\[
\LL_3=k\lp t\rp\bb e_1,e_2,\pi u_1, \pi u_2,u_3, v_1, v_2, \pi v_3\pp.
\]
They contain the para-unit element $e$, and are closed under multiplication. Meanwhile, we have 
$\LL_2\subsetneq \LL_2^{\vee}\subsetneq t^{-1}\LL_2$, and $\LL_2^{\vee}\star \LL_2^{\vee}\subset t^{-1}\LL_2$. The corresponding graded lattice chain 
\begin{equation}\label{eq 61}
	\cdots\subsetneq t\LL_2^{\vee}\subsetneq \LL_2\subsetneq \LL_2^{\vee}\subsetneq t^{-1}\LL_2\subsetneq \cdots
\end{equation}
is the vertex of type 2 in $\calB(G)$. The lattice $\LL_3$ satisfies $\LL_3\subsetneq \LL_3^{\vee}\subsetneq t^{-1}\LL_3$, and $M=t\LL_3^{\vee}\star \LL_3^{\vee}+\LL_3$ is self-dual. The corresponding graded lattice chain
\begin{equation}\label{eq 62}
\cdots\subsetneq t\LL_3^{\vee}\subsetneq \LL_3\subsetneq M\subsetneq \LL_3^{\vee}\subsetneq t^{-1}\LL_3\subsetneq\cdots	
\end{equation}
is the vertex of type 3 in $\calB(G)$. We denote by $\SSO(\LL_2)$ (resp. $\SSO(\LL_3)$) the subscheme of points in the isomorphism group scheme $\Isom(\LL_2)$ (resp. $\Isom(\LL_3)$) that preserve the form and whose determinant is equal to 1.

\begin{defi}\label{def 63}
	The parahoric subgroup $\scrG_2$ corresponding to vertices of type 2, is the stabilizer subgroup scheme of $\LL_2$ in $G(k\llp t\rlp)$. More precisely, $\scrG_2$ is an affine group scheme over $\Spec(k\lp t\rp)$ that represents the functor from $k\lp t\rp$-algebras to groups that sends $R$ to
	\[
	\scrG_2(R)=\{g\in \SSO(\LL_2)(R)\mid g(x\star y)=g(x)\star g(y)\},
	\]
	for all $x,y\in \LL_2\otimes_{k\lp t\rp}R$. Similarly, the parahoric subgroup $\scrG_3$ corresponding to vertices of type 3 represents the functor that sends $k\lp t\rp$-algebras $R$ to
	\[
	\scrG_3(R)=\{g\in \SSO(\LL_3)(R)\mid g(x\star y)=g(x)\star g(y)\},
	\]
	for all $x,y\in \LL_3\otimes_{k\lp t\rp}R$. 
\end{defi}

{\begin{rk}\label{rk 64}\rm
	For any $g\in \scrG_2(k\lp t\rp)$, we have $g(\LL_2)=\LL_2$. It is easy to see that $g(\LL_2^{\vee})=\LL_2^{\vee}$ since $\bb g(x),\LL_2\pp=\bb x, g^{-1}(\LL_2)\pp=\bb x,\LL_2\pp\in k\lp t\rp$. So that $g$ preserves the graded lattice chain (\ref{eq 61}). 
	
	For $g\in\scrG_3(k\lp t\rp)$, we get $g(\LL_3)=\LL_3$ and $g(\LL_3^{\vee})=\LL_3^{\vee}$ similarly as above. For $g(M)=tg(\LL_3^{\vee}\star \LL_3^{\vee})+g(\LL_3)$, since $g$ is a change of basis of $\LL_3^{\vee}$, we have $g(\LL_3^{\vee}\star \LL_3^{\vee})=\LL_3^{\vee}\star \LL_3^{\vee}$. Thus $g(M)=M$, and $g$ preserves the graded lattice chain (\ref{eq 62}).
\end{rk}
}

The quotient fpqc-sheaves:
\[
\Gr_{\scrG_2}=LG/L^+\scrG_2,~ (\text{resp.}~\Gr_{\scrG_3}=LG/L^+\scrG_3)
\]
are by definition the affine Grassmannians for $\scrG_2$ (resp. $\scrG_3$). We want to give explicit descriptions of affine Grassmannians $\Gr_{\scrG_2}$ (resp. $\Gr_{\scrG_2}$) in this section.

\begin{thm}\label{thm 65}
	$(a)$. There is an $LG$-equivariant isomorphism 
	\[
	\Gr_{\scrG_2}\simeq \calF_2
	\] 
	where the functor $\calF_2$ sends a $k$-algebra $R$ to the set of $R\lp t\rp$-lattices $L$ of $C_s \otimes_{k\llp t\rlp} R\llp t\rlp$ such that 
	\begin{itemize}
			\item [(1)] $L$ is an order in $C_s \otimes_{k\llp t\rlp} R\llp t\rlp$, i.e., $e\in L$, $L\star L\subset L$.
			\item [(2)] $L\subsetneq L^{\vee}\subsetneq t^{-1}L$.
			\item [(3)] $L^{\vee}\star L^{\vee}\subset t^{-1}L$. 
	    \end{itemize}
$(b)$. There is an $LG$-equivariant isomorphism 
	\[
	\Gr_{\scrG_3}\simeq \calF_3
	\] 
	where the functor $\calF_3$ sends a $k$-algebra $R$ to the set of $R\lp t\rp$-lattices $L$ of $C_s \otimes_{k\llp t\rlp} R\llp t\rlp$ such that 
	\begin{itemize}
			\item [(1)] $L$ is an order in $C_s \otimes_{k\llp t\rlp} R\llp t\rlp$.
			\item [(2)] $L\subsetneq L^{\vee}\subsetneq t^{-1}L$.
			\item [(3)] $M=tL^{\vee}\star L^{\vee}+L$ is self-dual. 
	    \end{itemize}
 \end{thm}

For any $k$-algebra $R$, we have $\scrF_2(R), \scrF_3(R)$ are not empty, as $\LL_2\otimes R\lp t\rp\in\scrF_2(R)$, $\LL_3\otimes R\lp t\rp\in\scrF_3(R)$. Note that $\calF_2(R)\cap\calF_3(R)=\emptyset$. Indeed, if $L\in\calF_2(R)$, we have $L^{\vee}\star L^{\vee}\subset t^{-1}L$, and $M=tL^{\vee}\star L^{\vee}+L=L$ is not self-dual. So that $L\notin\calF_3(R)$.

The proof of Theorem \ref{thm 65} is similar to what we did in \S \ref{AffGrass1}, i.e., we want to check the following two statements:
\begin{itemize}
	\item[(1)] For any $k$-algebra $R$, the subgroup of $LG(R)$ consisting of elements that stabilize the lattice $\LL_2 \otimes_{k\lp t\rp} R\lp t\rp$ (resp. $\LL_3 \otimes_{k\lp t\rp} R\lp t\rp$) agrees with $L^+\scrG_2 (R)$ (resp. $L^+\scrG_3 (R)$).
	\item[(2)] Let $L\in\calF_2(R)$ (resp. $L\in\calF_3(R)$). Then locally for the \etale topology on $R$, there exists $g\in LG(R)$ such that $L = g(\LL_2 \otimes_{k\lp t\rp} R\lp t\rp)$ (resp. $L = g(\LL_3 \otimes_{k\lp t\rp} R\lp t\rp)$).
\end{itemize}

Part (i) comes directly from Definition \ref{def 63}. To prove part (ii), observe that there are natural morphisms
\begin{equation}\label{mor56}
LG\rightarrow \scrF_2,\quad (\text{resp.}~ LG\rightarrow \scrF_3)	
\end{equation}
given by $g\mapsto g(\LL_2\otimes R\lp t\rp)$ (resp. $g\mapsto g(\LL_3\otimes R\lp t\rp)$). Set $L=g(\LL_2\otimes R\lp t\rp)$. We have $g((\LL_2\otimes R\lp t\rp)^\vee)=g(\LL_2\otimes R\lp t\rp)^\vee$ since $g$ preserves the form $\bb \ ,\ \pp$. So $L$ satisfies the condition (2), (3) in Theorem \ref{thm 65} (a) since $\LL_2\otimes R\lp t\rp \in\scrF_2(R)$ satisfies. We also have $e\in L$ and $L\star L\subset L$ since $g$ preserves the multiplication. Thus, $g(\LL_2\otimes R\lp t\rp)\in\scrF_2(R)$. Similarly, we can show that $g(\LL_3\otimes R\lp t\rp)\in\scrF_3(R)$. Thus, morphisms in (\ref{mor56}) are well-defined. By part (i), the stabilizer subgroups of lattices $\LL_2\otimes R\lp t\rp$ (resp. $\LL_3\otimes R\lp t\rp$) are $L^+\scrG_2(R)$ (resp. $L^+\scrG_3(R)$), so we obtain
\begin{equation}\label{mor57}
	LG/L^+\scrG_2\rightarrow \scrF_2,\quad (\text{resp.}~ LG/L^+\scrG_3\rightarrow \scrF_3)	.
\end{equation}

It is enough to show part (ii) for $(R,\calM)$ local strictly henselian. By \cite[Appendix, Lemma A.21]{RZ}, 
there exists a ``standard form" in a set of lattices 
$L$ satisfying $L\subsetneq L^{\vee}\subsetneq t^{-1}L$. Generally, let $V$ be an $F$-vector space with dimension $d=2n$ or $d=2n+1$ (we assume $d\geq 5)$, equipped with  a non-degenerate symmetric form $\bb  ~,~\pp$. Consider a lattice $L\subset V$ satisfying $L\subsetneq L^{\vee}\subsetneq t^{-1}L$ with $l+l'=d$. Here $l$ (resp. $l'$) is the length $l=lg(L^{\vee}/L)$ (resp. $l'=lg(t^{-1}L/L^{\vee}))$. 
Since octonion algebras have dimension 8, we only list results for even dimension case.

For $d=2n$,  
we can find a basis $\{x_i\}$ of 
$L$ as in one of the following cases:
\begin{itemize}
	\item [(1)] Split form: when $l=2r$, we have 
	$L=\oplus_{i=1}^d\calO_{F}\cdot x_i$ with 
	\begin{flalign*}
		\bb x_i,x_{d+1-j}\pp &=\delta_{ij}, ~\text{for}~ i\notin [n-r+1,n+r];\\
		\bb x_i,x_{d+1-j}\pp &=\pi \delta_{ij}, ~\text{for}~ i\in [n-r+1,n+r].
	\end{flalign*}
	\item [(2)] 	Quasi-split form: when $l=2r+1$, we have 
	$L_=\oplus_{i=1}^d \calO_{F}\cdot  x_i$ with 
	\begin{flalign*}
		&\bb x_i,x_{d+1-j}\pp =\delta_{ij}, ~\text{for}~ i\notin [n-r,n+r+1];\\
		&\bb x_i,x_{d+1-j}\pp =\pi\delta_{ij}, ~\text{for}~ i\in [n-r,n+r+1]/\{n,n+1\};\\
		&\bb x_n,x_n\pp=\pi, \quad \bb x_{n+1},x_{n+1}\pp=1,\quad \bb x_n,x_{n+1}\pp=0.
	\end{flalign*}
\end{itemize}

Here $\pi$ is a uniformizer of $\calO_{F}$. 
In our case, $\calO_{F}=k\lp t \rp$ and $L$ is a lattice in  $C_s$ with $d=8$. For $L\in \calF_2(R)$ or $L\in \calF_3(R)$, we have $L\subsetneq L^{\vee}\subsetneq t^{-1}L$. 
We claim that $L\subset C_s$ is of split form. 

Recall that the discriminant $dL$ of $(L,\bb ~, ~\pp)$ is the determinant of the matrix with respect to the symmetric form $\bb~,~\pp$. By change of basis in $L$, two matrices are congruent to each other with respect to different basis (i.e., $B'=T^TBT$, where $T$ is the change of basis matrix), so we have 
\[
dL\in k\lp t\rp^*/k\lp t\rp^{*2}.
\]
The discriminant is a class invariant of quadratic forms. From above, for a lattice $L\subset V$, with $\dim V=2n$. If $L$ is of split form, we have $dL=1$; if $L$ is of quasi-split form, $dL=t$. In our case, $L$ is a lattice of the split para-octonion algebra $C_s$. Since $C_s$ can be viewed as an orthogonal sum of hyperbolic subspaces, we have $dC_s=1$. Thus, if $L$ is of quasi-split form, we have $L\otimes_{k\lp t\rp}k\llp t\rlp\simeq C_s$, and 
\[
d(L\otimes_{k\lp t\rp}k\llp t\rlp)=t\neq 1\in k\llp t\rlp^*/k\llp t\rlp^{*2},
\]
which is a contradiction. Therefore, we only need to consider the split form case. More precisely: we consider the cases when $l=2,4,6$. Here $L=\oplus_{i=1}^8 k\lp t\rp x_i$, and the symmetric forms $\bb~,~\pp$ with respect to basis $\{x_i\}$ are:
	\begin{flalign*}
		&\bb~,~\pp={\rm antidiag}(1,1,1,t,t,1,1,1),\quad\text{if}~ l=2;\\
		&\bb~,~\pp={\rm antidiag}(1,1,t,t,t,t,1,1),\quad\text{if}~ l=4;\\
        &\bb~,~\pp={\rm antidiag}(1,t,t,t,t,t,t,1),\quad\text{if}~ l=6.
	\end{flalign*}
	Our main conclusion is the following:
\quash{
 Quasi-split form: when $l=1,3,5,7$, we have $L=\oplus_{i=1}^8 R\lp t\rp x_i$. Symmetric forms $\bb~,~\pp$ with respect to basis $\{x_i\}$ are:
	\[
	\left(\begin{array}{ccccccc}
		&&&&&&i_1\\
		&&&&&i_2&\\
		&&&&i_3&&\\
		&&&\left(\begin{array}{cc}
		t&0\\
		0&1\\
		\end{array}
		\right)&&&\\
		&&i_3&&&&\\
		&i_2&&&&&\\
		i_1&&&&&&
	\end{array}
	\right),
	\]
	where 
	\begin{flalign*}
		&i_1=i_2=i_3=1, \quad \text{for}~l=1;\\
		&i_1=i_2=1,\quad i_3=t, \quad \text{for}~l=3;\\
		&i_1=1,\quad i_2=i_3=t, \quad \text{for}~l=5;\\
		&i_1=i_2=i_3=t, \quad \text{for}~l=7.
	\end{flalign*}
}
 
\begin{itemize}
	\item when $l=2$, we claim that $L$ is not closed under multiplication ($L\star L\nsubseteq L$). Hence $L\notin \calF_2(R)$ and $L\notin \calF_3(R)$.
	
	\item when $l=4$, there exists a basis of $L$, such that its multiplication table is the same as multiplication table in $\LL_2 \otimes R\lp t\rp$. Hence there exists $g\in LG(R)$, such that $L=g(\LL_2 \otimes R\lp t\rp)$. Then we get:
	    \begin{flalign*}
		&L\star L\subset L;\\
		&L^{\vee}\star L^{\vee}\subset t^{-1}L,
	\end{flalign*}
	since $L^\vee=g((\LL_2\otimes R\lp t\rp)^\vee)$. So that $L\in\calF_2(R)$. We also get $L\notin\calF_3(R)$ since $\calF_2(R)\cap\calF_3(R)=\emptyset$.
	\item when $l=6$, there exists a basis of $L$ such that its multiplication table is the same as  table in $\LL_3 \otimes R\lp t\rp$. Similarly, we have $L=g(\LL_3 \otimes R\lp t\rp)$ for some $g\in LG(R)$ and 
	\begin{flalign*}
		&L\star L\subset L;\\
		&M=tL^\vee\star L^\vee+L ~\text{is self-dual},
	\end{flalign*}
since $M=g(t(\LL_3^\vee \otimes R\lp t\rp)\star(\LL_3^\vee \otimes R\lp t\rp)+(\LL_3 \otimes R\lp t\rp))$. So that $L\in\calF_3(R)$
\end{itemize}

For any $L\in\calF_2(R)$ (resp. $L\in\calF_3(R)$), consider $l=lg(L^\vee/L)$. Here $l$ can only be equal to $2, 4, 6$ since $L$ is of split form. From above claims, we get $l=4$ (resp. $l=6$). Thus, there exists $g\in LG(R)$ such that $L=g(\LL_2 \otimes R\lp t\rp)$ (resp. $L=g(\LL_3 \otimes R\lp t\rp)$). This finish
    the proof of Theorem \ref{thm 65}. 

The proof of claims is similar to \S \ref{AffGrass1}. Let $L$ be a lattice with $e\in L$, $L\subsetneq L^{\vee}\subsetneq t^{-1}L$. No matter $l=lg(L^\vee/L)=2,4$ or $6$, there exists a basis $\{x_i\}$ of $L$, such that the symmetric form is an anti-diagonal matrix with respect to $\{x_i\}$. Consider $e=\sum_{i=1}^8 a_ix_i\in L$. We set:
\[
e_1=\sum_{i=1}^4 a_ix_i,\quad e_2=\sum_{i=5}^8 a_ix_i.
\]
Then $q(e_1)=q(e_2)=0$. We get $\bb e_1,e_2\pp=1$ by $q(e)=1$, and we also have $e_1\star e_1=e_2$, $ e_2\star e_2=e_1$, $e_1\star e_2=e_2\star e_1=0$ by similar calculation that we did in \S \ref{AffGrass1}. Thus, $R\lp t\rp e_1\oplus R\lp t\rp$ is a hyperbolic subspace. By Nakayama's lemma, $\{e_1, e_2\}$ are primitive elements in $L$. Choose elements $y_3,\cdots,y_8\in L$ such that $\{e_1, e_2,  y_3,..., y_8\}$ is a basis of $L$, By replacing $y_i$ to 
$
y_i'=y_i-\bb y_i,e_2\pp e_1-\bb y_i,e_1\pp e_2,
$
we can always assume that $y_i\in L_0$, where
\[
L_0=\{x\in L\mid \bb x,e_1\pp=\bb x,e_2\pp=0\}.
\]
Therefore, $L=R\lp t\rp e_1\oplus R\lp t\rp e_2 \oplus L_0$. Set $L_i=L_0\star e_i$ for $i=1,2$. 

\begin{lm}\label{lm 66}
	We have 
	\begin{itemize}
			\item [(1)] $L_i\subset L_0$ for $i=1,2$. 
			\item [(2)] $e_i\star L_i=0,~ L_i\star e_{i+1}=0$ for $i=1,2\mod 2$.
		\end{itemize}
\end{lm}

We omit the proof of Lemma \ref{lm 66}, since it is the same as the proof of Lemma 5.1. For any $x\in L_0$, we get
\[
x=\bb e_1,e_2\pp x=(e_2\star x)\star e_1+(e_1\star x)\star e_2\in L_1+L_2,
\]
by Lemma \ref{lm 28}. Hence $L_0=L_1+L_2$. We obtain $L_0=L_1\oplus L_2$	since $L_1, L_2$ are totally isotropic with rank $\leq 3$. When we restrict the symmetric form to $L_0$, there exists a basis $\{u_i\}_{i=1}^3$ of $L_1$ (resp. $\{v_i\}_{i=1}^3$ of $L_2$) such that 
\[
\bb u_i,v_j\pp=0,
\]
for $i\neq j$, since the symmetric form is an anti-diagonal matrix.

\begin{lm}\label{lm 67}
We have

$(1)$ \begin{flalign*}
		&u_i\star e_1=e_2\star u_i=-u_i;\\
		&v_i\star e_2=e_1\star u_i=-v_i,
	\end{flalign*}
	
$(2)$ \begin{flalign*}
		&u_i\star v_j=-\bb u_i,v_j\pp  e_1;\\
		&v_j\star u_i=-\bb u_i,v_j\pp  e_2,
	\end{flalign*}
	
$(3)$ 
	\[
	u_i\star u_j=-u_j\star u_i\in L_2,\quad v_i\star v_j=-v_j\star u_i\in L_1.
	\]
for $i,j=1,2,3$.	
\end{lm}

\begin{proof}
The proof of Lemma \ref{lm 67} (1) is similar to Lemma \ref{lm 52}. For (2), we have
\begin{flalign*}
	u_i\star v_j&=-(u_i\star e_1)\star v_j
	            =(v_j\star e_1)\star u_i-\bb u_i,v_j\pp e_1\\
	            &=-\bb u_i,v_j\pp e_1.
\end{flalign*} 
Similarly, we also have $v_j\star u_i=-\bb u_i,v_j\pp  e_2$. For (3), it is easy to see that $u_i\star u_j\in L_0$ since $\bb u_i\star u_j.e_1\pp=\bb u_i\star u_j,e_2\pp=0$. Consider
\[
(e_2\star u_i)\star u_j+(u_j\star u_i)\star e_2=\bb e_2,u_j\pp u_i=0,
\]
and $(e_2\star u_i)\star u_j=-u_i\star u_j$ by (1). We get $u_i\star u_j=(u_j\star u_i)\star e_2\in L_0\star e_2=L_2$. Thus, $(u_j\star u_i)\star e_2=-u_j\star u_i$ since $v_i\star e_2=-v_i$, which gives us $u_i\star u_j=-u_j\star u_i$. Similarly, we have $v_i\star v_j=-v_j\star v_i$.
\end{proof}
	
So far we only use the fact that the symmetric form is an antidiagonal matrix. To prove our claims, we now need to consider the explicit number of $l=lg(L^\vee/L)$.

{\bf (a) $l=2$ case}. When $l=2$, the restriction of $\bb ~,~\pp$ to $L_0$ is :
\[
\bb ~,~\pp \mid_{L_0}={\rm antidiag}(1,1,t,t,1,1).
\]
Since $L_0=L_1\oplus L_2$ and $L_1, L_2$ are totally isotropic, there exists a basis $\{u_i\}_{i=1}^3$ of $L_1$ (resp. $\{v_i\}_{i=1}^3$ of $L_2$) such that 
\begin{flalign*}
	&\bb u_1,v_1\pp=t,\quad \bb u_2,v_2\pp=1;\\
	&\bb u_3,v_3\pp=1,\quad \bb u_i,v_j\pp=0,
\end{flalign*}
for $i\neq j$. By Lemma \ref{lm 66}, Lemma \ref{lm 67}, we almost complete the multiplication table of $L$. The only left part is $u_i\star u_j$ (resp. $v_i\star v_j$).

Consider $u_1\star u_2\in L_2$. We obtain 
\[
\bb u_1\star u_2, u_1\pp=-\bb u_1\star u_2, u_1\star e_1\pp=q(u_1)\bb u_2,e_1\pp=0,
\]	
by Lemma 2.7	. Similarly, we get $\bb u_1\star u_2, u_2\pp=0$. So we can assume $u_1\star u_2=\lambda v_3$. Suppose that $L\star L\subset L$. We have $\lambda\in R\lp t\rp$. Since
\[
(u_1\star u_2)\star v_1+(v_1\star u_2)\star u_1=\bb u_1,v_1\pp u_2=tu_2,
\]
and $v_1\star u_2=0$ by Lemma \ref{lm 67} (2), we get $(u_1\star u_2)\star v_1=tu_2$. Thus $\lambda(v_3\star v_1)=tu_2$. Similarly, consider $(u_1\star u_2)\star v_2=\lambda(v_3\star v_2)$ and $(u_1\star u_2)\star v_3=\lambda(v_3\star v_3)$, we get
\begin{equation}\label{eq510}
\lambda(v_3\star v_2)=-u_1,\quad v_3\star v_3=0.
\end{equation}
Thus, $\lambda^{-1}\in R\lp t\rp$, which implies $\lambda\in R\lp t\rp^*$. We can assume $\lambda=1$ and get
\begin{equation}\label{eq511}
v_3\star v_1=tu_2,\quad v_3\star v_2=-u_1,\quad v_3\star v_3=0.	
\end{equation}
Next we consider $u_2\star u_3\in L_2$. We can assume $u_2\star u_3=\mu v_1$ for some $\mu\in R\lp t\rp$, since $\bb u_2\star u_3, v_2\pp=\bb u_2\star u_3, v_3\pp=0$. From Lemma \ref{lm 67} (1), (2), we deduce that
\[
(u_2\star u_3)\star v_3+(v_3\star u_3)\star u_2=\bb u_2,v_3\pp u_3,
\]
and $(v_3\star u_3)\star u_2=u_2$, $\bb u_2,v_3\pp=0$. Thus, we have $\mu(v_1\star v_3)=-u_2$. But $v_1\star v_3=-v_3\star v_1=-tu_2$ from equation (\ref{eq510}). Therefore
\[
-t\mu u_2=-u_2.
\]
Hence $\mu=t^{-1}\notin R\lp t\rp$, which is a contradiction. So that $L\star L\nsubseteq L$. \\

{\bf (b) $l=4$ case}. When $l=4$, the restriction of $\bb ~,~\pp$ to $L_0$ is :
\[
\bb ~,~\pp \mid_{L_0}={\rm antidiag}(1,t,t,t,t,1).
\]
There exists a basis $\{u_i\}_{i=1}^3$ of $L_1$ (resp. $\{v_i\}_{i=1}^3$ of $L_2$) such that 
\begin{flalign*}
	&\bb u_1,v_1\pp=t,\quad \bb u_2,v_2\pp=t;\\
	&\bb u_3,v_3\pp=1,\quad \bb u_i,v_j\pp=0,
\end{flalign*}
for $i\neq j$. Assume $L\star L\subset L$. By using the same discussion as we did in $l=2$ case, we get the following multiplication table:\\

	\begin{minipage}{\textwidth}
		\begin{minipage}[t]{0.45\textwidth}
			\centering
			\makeatletter\def\@captype{table}\makeatother\caption{$u_i*u_j$}
			\vspace{-3mm}
			\begin{tabular}{|c|c|c|c|}\hline
				$*$& $u_1$&$u_2$&$u_3$\\
				\hline
				$u_1$& $0$& $tv_3$ & $-v_2$\\
				\hline
				$u_2$& $-tv_3$& 0&$v_1$\\
				\hline
				$u_3$& $v_2$& $-v_1$&0\\
				\hline
			\end{tabular}
		\end{minipage}
		\begin{minipage}[t]{0.45\textwidth}
			\centering
			\makeatletter\def\@captype{table}\makeatother\caption{$v_i*v_j$}
			\vspace{-3mm}
			\begin{tabular}{|c|c|c|c|}\hline
				$*$& $v_1$&$v_2$&$v_3$\\
				\hline
				$v_1$& 0& $tu_3$ & $-u_2$\\
				\hline
				$v_2$& $-tu_3$& 0&$u_1$\\
				\hline
				$v_3$& $u_2$& $-u_1$&0\\
				\hline
			\end{tabular}
		\end{minipage}
	\end{minipage}\\

It is easy to see that $L$ has the same multiplication table as $\LL_2\otimes R\lp t\rp$. Thus, there exists $g\in LG(R)$, such that $L=g(\LL_2 \otimes R\lp t\rp)$. 
So that $L\in\calF_2(R)$.\\

{\bf (c) $l=6$ case}. When $l=6$, the restriction of $\bb ~,~\pp$ to $L_0$ is :
\[
\bb ~,~\pp \mid_{L_0}={\rm antidiag}(t,t,t,t,t,t).
\]
There exists a basis $\{u_i\}_{i=1}^3$ of $L_1$ (resp. $\{v_i\}_{i=1}^3$ of $L_2$) such that 
\begin{flalign*}
	&\bb u_1,v_1\pp=t,\quad \bb u_2,v_2\pp=t;\\
	&\bb u_3,v_3\pp=t,\quad \bb u_i,v_j\pp=0,
\end{flalign*}
for $i\neq j$. Assume $L\star L\subset L$. Again, by using the same discussion as we did in $l=2$ and $l=4$ case, we get the following multiplication table:\\

	\begin{minipage}{\textwidth}
		\begin{minipage}[t]{0.45\textwidth}
			\centering
			\makeatletter\def\@captype{table}\makeatother\caption{$u_i*u_j$}
			\vspace{-3mm}
			\begin{tabular}{|c|c|c|c|}\hline
				$*$& $u_1$&$u_2$&$u_3$\\
				\hline
				$u_1$& $0$& $tv_3$ & $-tv_2$\\
				\hline
				$u_2$& $-tv_3$& 0&$tv_1$\\
				\hline
				$u_3$& $tv_2$& $-tv_1$&0\\
				\hline
			\end{tabular}
		\end{minipage}
		\begin{minipage}[t]{0.45\textwidth}
			\centering
			\makeatletter\def\@captype{table}\makeatother\caption{$v_i*v_j$}
			\vspace{-3mm}
			\begin{tabular}{|c|c|c|c|}\hline
				$*$& $v_1$&$v_2$&$v_3$\\
				\hline
				$v_1$& 0& $u_3$ & $-u_2$\\
				\hline
				$v_2$& $-u_3$& 0&$u_1$\\
				\hline
				$v_3$& $u_2$& $-u_1$&0\\
				\hline
			\end{tabular}
		\end{minipage}
	\end{minipage}\\

Thus $L$ has the same multiplication table as  $\LL_3 \otimes R\lp t\rp$. Similarly, there exists $L=g(\LL_3 \otimes R\lp t\rp)$ for some $g\in LG(R)$, and we have
\[
M=tL^\vee\star L^\vee+L ~\text{is self-dual}.
\]
So that $L\in\calF_3(R)$.

Therefore, for $L\in\calF_2(R)$, we show that $l$ is equal to $4$ (otherwise, if $l=2$, $L\star L\nsubseteq L$; if $l=6$, $L\in\calF_3(R)$ implies $L\notin \calF_2(R)$). There exists $g\in LG(R)$, such that $L=g(\LL_2 \otimes R\lp t\rp)$. For $L\in\calF_3(R)$, we show that $l=6$, and there exists $g\in LG(R)$, such that $L=g(\LL_3 \otimes_{k\lp t\rp} R\lp t\rp)$. This finish the proof of Theorem \ref{thm 65}.
	
\bigskip	
	
	
	\bibliographystyle{plain}
	\let\itshape\upshape
	\bibliography{ref}


	\Addresses

\end{document}